\documentclass[11pt]{article}
\usepackage{mathrsfs}
\usepackage{amsfonts,amsmath,dsfont,amssymb,amsthm,stmaryrd,bbm,color}
\usepackage[hidelinks]{hyperref}
\usepackage{bbm}

\newtheorem{thm}{Theorem}[section]
\newtheorem{prop}[thm]{Proposition}

\newtheorem{lem}[thm]{Lemma}

\newtheorem{rmk}{Remark}

\newcommand{\ind}{\mathbbm{1}}

\newcommand{\maxload}{\mathrm{MaxLoad}}
\newcommand{\cho}{\mathrm{Chosen}}

\newcommand{\N}{\mathbb{N}}

\newcommand{\E}{\mathbb{E}}

\def\P{{\mathbb P}}

\begin{document}
\title{Load balancing under  $d$-thinning
\thanks{Research supported by ISF grant 1707/16}
}
\author{Ohad N. Feldheim,
Jiange Li \thanks{Both authors are with the Einstein Institute of Mathematics at the Hebrew University of Jerusalem.
E-mail: {\tt ohad.feldheim@mail.huji.ac.il,  jiange.li@mail.huji.ac.il}.}
}
\date{\today}
\maketitle

\begin{abstract}
In the classical balls-and-bins model, $m$ balls are allocated into $n$ bins one by one uniformly at random.
In this note, we consider the $d$-thinning variant of this model, in which the process is regulated in an on-line fashion as follows. For each ball, after a random bin has been selected, an overseer may decide, based on all previous history, whether to accept this bin or not. However, one of every $d$ consecutive suggested bins must be accepted. The \emph{maximum load} of this setting is the number of balls in the most loaded bin.
We show that after $\Theta(n)$ balls have been allocated, the least maximum load achievable with high probability is $(d+o(1))\sqrt[d]{\frac{d\log n}{\log\log n}}$. This should be compared with the related $d$-choice setting, in which the optimal maximum load achievable with high probability is $\frac{\log\log n}{\log d}+O(1)$.
%the number of alternative choices offered to the overseer merely affects the achievable maximum load by a multiplicative constant.

%In the $d$-thinning model for task allocation, whenever a task arrives, an overseer is given a sequence of independent random slots uniformly chosen from $\{1,\dots, n\}$. After a slot is generated, the overseer may decide to reject it, thus thinning the sequence, subject to that at least one of every $d$ consecutive options provided must be accepted. The overseer aims to reduce the maximum load of the slots $\{1,\dots, n\}$ after $\rho n$ tasks being allocated, where $\rho$ is some constant. In this note, we show that if the overseer makes optimal decisions, the maximum load of order $(d+o(1))(d\log n/\log\log n)^{1/d}$ can be achieved with high probability. This confirms a conjecutre of the first author and Gurel-Gurevich (2018). This is also in contrast with related models in which the number of alternative slots offered to the overseer has no impact on the asymptotic maximum load up to a multiplicative constant.
\end{abstract}
%At each time, one ball is allocated into a bin uniformly random chosen from $[n]=\{1, 2, \cdots, n\}$. At time $t\in\N$, we denote by $h(t)$ the height of the $t$-th ball, and by $L_i(t)$ the load of the $i$-th bin. For $l\in\N$, we denote by $\mu_l(t)=|\{1\leq i\leq t: h(t)\geq l\}|$ the number of bins with load at least $l$ at time $t$.

\section{Introduction}
In the classical \emph{balls-and-bins} model, $m$ balls are sequentially allocated into $n$ random bins, each selected uniformly and independently at random. This model has been extensively studied in probability theory, random graph theory, and computer science, and it has found many applications in various areas, such as hashing, load balancing, and resource allocation in parallel and distributed systems (see  e.g., \cite{ABKU99}, \cite{KLM96}, \cite{Kor97}, \cite{SEK03}, \cite{Ste96}).

The \emph{maximum load} is defined as the number of balls in the fullest bin. When $m\geq n\log n$, the maximum load is known to be of $\frac{m}{n}+\Theta\Big(\sqrt{\frac{m\log n}{n}}\Big)$ with high probability (see, e.g., \cite{RS98}). In \cite{ABKU99}, Azar, Broder, Karlin and Upfal studied the \emph{two-choice} variant of the balls-and-bins model, in which for the allocation of each ball, an overseer has the freedom of making a choice between two bins, each selected independently and uniformly at random. They showed that for $m=\Theta(n)$, the greedy strategy of selecting the less loaded bin yields, with high probability, an optimal maximum load of $\frac{\log\log n}{\log 2}+O(1)$, an exponential improvement over the classical balls-and-bins model. Since then, this surprising phenomenon has found many algorithmic applications (see e.g. \cite{M01} and references therein). %This exponential improvement of the two-choice model has found many algorithmic applications. We refer to the classical survey \cite{M01} on this topic. %In the heavily loaded case $m\geq n\log n$, Berenbrink, Czumaj, Steger and V\"{o}cking \cite{BCSV06} showed that the maximum load of the greedy strategy is $m/n+\log\log n/\log 2$.

In~\cite{DFGR}, with motivations from statistics, Dwivedi, the first author, Gurel-Gurevich and Ramdas considered the following {\em two-thinning} variant of the balls-and-bins model. In this setting, for each ball, after a uniformly random \emph{primary} allocation has been suggested, an overseer may choose to either accept it, or allocate the ball into a new independent and uniformly random \emph{secondary} allocation. Note that the overseer is oblivious to the secondary allocation before deciding whether to accept the primary allocation or not. Hence, the overseer has less choice in this variant than in the two-choice setting. Equivalently, the two-thinning setting could be described as allowing the overseer to discard allocations on-line as long as no two consecutive allocations are rejected (thus justifying the term two-thinning).
This setup arises naturally in a statistical scenario in which one collects samples one-by-one and is allowed to decide whether to keep each sample or not, under the constraint that no two consecutive samples are discarded.
In~\cite{DFGR}, the authors showed that the power of two-thinning could reduce the discrepancy of a sequence of random points independently and uniformly chosen from the interval $[0,1]$ to be near optimal. The authors also show that a  weaker $(1+\beta)$-thinning setting, where discarding is allowed with probability $\beta<1$ independently for every sample whose predecessor was not discarded, also obtains near optimal discrepancy.

In~\cite{FGG18}, the first author and Gurel-Gurevich show that by setting a threshold and accepting the primary allocation if a ball is assigned to a bin in which the number of balls that already accepted their primary allocations is below this threshold, the overseer can achieve, with high probability, an optimal maximum load of $(2+o(1))\sqrt{\frac{2\log n}{\log\log n}}$, a polynomial improvement over the one-choice setting. Observe that the two-choice setting with errors (i.e., with constant probability of placing a ball into the unselected bin) and the $(1+\beta)$-thinning setting provide no asymptotic improvement over the one-choice setting (this follows from \cite{PTW15}, see also \cite{DFGR}).

In  \cite{ABKU99}, the $d$-choice setting, where the overseer is given $d>2$ choices, was also considered. In this setting, the greedy algorithm yields an optimal  maximum load of $\frac{\log\log n}{\log d}+O(1)$ with high probability; that is, compared with the case $d=2$, the performance only improves by a multiplicative factor for larger values of $d$. %the number of alternative choices merely affects the maximum load by a multiplicative constant.

This paper investigates the analogous \emph{$d$-thinning} setting for $d>2$, in which the overseer is allowed to reject up to $d-1$ consecutive allocations for each ball  (see an equivalent description in the abstract).
Our result is about the recovery and analysis of the optimality of the $\ell$-threshold strategy. This strategy rejects the $i$-ary allocation of a ball if  this ball is assigned to a bin which already contains $\ell$ balls whose $i$-ary allocations have been accepted. More details on the $\ell$-threshold strategy and its optimality are given in the next section. Our main result is the following.
\begin{thm}\label{thm:main}
Let $f$ be the $\Big(\frac{d\log n}{\log\log n}\Big)^{1/d}$-threshold strategy for the $d$-thinning of $\lfloor\rho n\rfloor$ balls into $n$ bins. Then $f$ is asymptotically optimal, and with high probability yields
$$\maxload^f(\lfloor\rho n\rfloor)=(d+o(1))\left(\frac{d\log n}{\log\log n}\right)^{1/d}.$$
\end{thm}
Clearly, the exponent of the maximum load in the $d$-thinning setting keeps improving as $d$ grows, but it is never asymptotically better than that in the two-choice setting.

Interestingly, our threshold strategy was used by Adler, Chakrabarti, Mitzenmacher and
Rasmussento \cite{ACMR98} to obtain a similar tight upper bound (up to a constant factor) in the related model of parallel allocation (which is neither stronger nor weaker than the setting considered here). %There, each ball is given $d$ alternative allocations.
%This is followed by $d-1$ rounds of parallel communication between each ball and the bins provided as its possible allocations. After this, each ball must independently decide on its allocation.
Nevertheless, we provide here a finer analysis of the threshold strategy and achieve an upper bound which is tight up to the $1+o(1)$ factor.

The more challenging part is to prove the matching lower bound on $\maxload^f(\lfloor\rho n\rfloor)$, which establishes the optimality of the $\Big(\frac{d\log n}{\log\log n}\Big)^{1/d}$-threshold strategy.  This is achieved in Section~\ref{sec: lower}, not by directly extending the argument in \cite{FGG18}, but rather via
a reduction lemma, which reduces a $d$-thinning problem to a $(d-1)$-thinning problem. After $d-1$ rounds of iteration, the original $d$-thinning problem is reduced to a one-thinning problem, which is an easy to handle problem in the classical balls-and-bins model.

\section{Notation and definitions}
%Throughout we consider the allocation of $\lfloor\rho n\rfloor$ balls into $n$ bins in the $d$-thinning model.

\textbf{Representation of the setting.}
Let $\{Z_1(r)\}_{r\in\N}$,$\cdots$,$\{Z_d(r)\}_{r\in\N}$ be $d$ sequences of independent random variables, which are uniformly distributed on the set $[n]:=\{1, 2, \cdots, n\}$. Here $\{Z_1(r)\}_{r\in\N}$ is used as a pool of primary allocations,
$\{Z_2(r)\}_{r\in\N}$ -- as a pool of secondary allocations and, in general,
$\{Z_d(r)\}_{r\in\N}$ -- as a pool of $d$-ary allocations.
%Here, $Z_1(r)$ represents the primary allocation of the $r$-th ball, while $\{Z_2(r)\}_{r\in\N}, \cdots, \{Z_d(r)\}_{r\in\N}$ are used as the pool of secondary, tertiary allocations and so on.
Elements are drawn from these pools one by one as needed, so that if the overseer accepts the primary allocation $Z_1(1)$ of the first ball and rejects the primary allocation $Z_1(2)$ of the second ball, then the bin $Z_2(1)$ is offered as the secondary allocation for the second ball.

\textbf{A $d$-thinning strategy.} A {\em $d$-thinning strategy} $f$ is a sequence of Boolean functions $f_{t,i}:[n]^{t}\to \{0,1\}$ with $i\in [d-1]$, which takes the final allocations of the first $t-1$ balls and the $i$-ary allocation of the $t$-th ball, and returns $0$ if this $i$-ary allocation is accepted and $1$ if it is rejected. We also allow $f_{t,i}$ to implicitly depend on any external, potentially random data, as long as this data is independent from all future allocations. For completion we also define $f_{t, d}\equiv 0$, which corresponds to the fact that in our setting, the $d$-ary allocation cannot be rejected. Let $\{Z(t)\}_{t\in\N}$ denote final allocations assigned by $f$.
%
%Let $\Omega$ be some sample space.
% A (probabilistic) {\em $d$-thinning strategy} $f$ is a sequence of Boolean functions
%  $$f_{t,i,\omega}:[n]^{t}\times\Omega\to \{0,1\}$$ with $i<d$, so that $f_{t,i,\omega}$ takes the
%  final allocations of the first $t-1$ balls, the $i$-ary allocation
%  of the $t$-th ball and an element in $\Omega$ used as an external source of randomness. The
%  strategy returns $0$ if this $i$-ary allocation is accepted  and $1$ if it is rejected.
%  %Generally we omit the third coordinate when a statement hold for any $\omega\in \Omega$,
%
%

Towards obtaining a formal description of $Z(t)$, set $r_1(t)=t$ and for $i\in\{2, \cdots, d\}$ inductively define
\begin{equation}\label{eq:r_i}
r_i(t)=\sum_{k=1}^t \prod_{j=1}^{i-1}f_{k,j}(Z(1),\dots,Z(k-1),Z_j(r_j(k))).
\end{equation}
Thus, $r_i(t)$ represents how many of the first $t$ balls have their first $i-1$ rounds of allocations rejected.  For $i\in[d]$ and $r\in\N$, we define
\begin{equation}\label{eq:t_i(r)}
t_i(r)=\min\{t\in\N: r_{i}(t)=r\},
\end{equation}
which is the unique time at which $Z_{i}(r)$ is used as an $i$-ary allocation. For $t\in\N$, the function $\cho(t)$ determines which round allocation is accepted for the $t$-th ball, i.e.,
$$
\cho(t)=\min\big\{i\in[d]: f_{t,i}(Z(1),\dots,Z(t-1),Z_i(r_i(t)))=0\big\}.
$$
The \emph{final allocation} assigned by $f$ is now defined via
$$
Z(t)=Z_{\cho(t)}(r_{\cho(t)}(t)).
$$

\textbf{Induced thinning strategy.}
For $j\in[d]$, let  $f^j$ be the $(d-j+1)$-thinning strategy induced by $f$ as follows
\begin{equation}\label{fj def}
f^j_{r,i}=f_{t_j(r),i+j-1}.
\end{equation}
Particularly, we have $f^1=f$. The thinning strategy  $f^j$ accepts or rejects the $i$-ary allocation of the $r$-th ball (with respect to $f^j$) according to the decision of $f$ on the $(i+j-1)$-ary allocation of the $t_j(r)$-th ball (with respect to $f$), i.e., the $r$-th ball whose first $j-1$ allocations we rejected.
Observe that the $(d-j+1)$-thinning strategy $f^j$, relies on information concerning the first $j-1$ rounds of allocations assigned by $f$. This is still a valid strategy as these
allocations are independent of future allocations from $\{Z_k(r_k(t))\}_{k> j}$ and $\{Z_k(r)\}_{k\in [d],r> r_k(t)}$.

We also introduce the operation
\begin{equation}\label{eq:fplus def}
f^+=f^2,
\end{equation}
and observe that for all $j\in[d-1]$ we have $(f^j)^+=f^{j+1}$.

\textbf{Load analysis notation.} For $m\in[n]$ and $t\in\N$,  the load of bin $m$ after $t$ balls have been allocated is defined as
\begin{align}\label{eq:load-m}
L_m^f(t)=|\{1\leq k\leq t: Z(k)=m\}|.
\end{align}
Let $L_{i, m}^f(r)$ be the number of balls that accept bin $m$ as their $i$-ary allocations after $t_i(r)$ balls were allocated, i.e.,
\begin{align}\label{eq:load-i-m}
L_{i, m}^f(r)=|\{1\leq k\leq t_i(r): \cho(k)=i, Z(k)=m\}|.
\end{align}
Clearly, we have $\sum_{i=1}^dL_{i,m}^f(r_i(t))= L_{m}^f(t)$. The maximum load over a subset $S\subseteq [n]$ after $t$ balls have been allocated using the strategy $f$ is defined as
\begin{align}\label{eq:max-load}
\text{MaxLoad}_S^f(t)=\max_{m\in S}L_m^f(t).
\end{align}
In addition, we define
\begin{align}\label{eq:super-level}
\phi_S^f(t)=\big|\big\{m\in S: L_m^f(t)>0\big\}\big|,
\end{align}
which is the number of non-empty bins in $S$ after $t$ balls were allocated, and
\begin{align}\label{eq:primary-super-level}
\psi_S^f(t)=|\{m\in S: \text{there exists}~ 1\leq k\leq t~ \text{such that} ~Z_1(k)=m\}|,
\end{align}
which counts the number of bins in $S$ offered as primary allocations in the process of allocating of $t$ balls. For simplicity, we drop the subscript $S$ when $S=[n]$.

\textbf{Asymptotic optimality.} We say that a $d$-thinning strategy $f$ is \emph{asymptotically optimal} if for any other $d$-thinning strategy $g$, with high probability, as $n$ tends to infinity we have
$$\text{MaxLoad}^f(t)\le (1+o(1))\cdot\text{MaxLoad}^g(t).$$

\textbf{Threshold strategy.} The $\ell$-threshold strategy for $\ell>0$ is the strategy $f$ that accepts an allocation when the suggested bin contains no more than $\ell$ balls that were allocated in the $i$-ary allocations round, i.e.,
$$
f_{t,i}\big(Z(1),\dots,Z(t-1),Z_i(r_{i}(t))\big)=\ind\left(L^f_{i, Z_i(r_{i}(t))}(r_{i}(t))> \ell\right),
$$
where $\ind(E)$ is the indicator function for the event $E$.

The following lemmata from \cite{FGG18} and \cite{MU05} will be of use in our analysis. The first is a comparison bound relating the balls-and-bins model to independent Poisson random variables. We denote by $\N_0=\N\cup\{0\}$. Given $x, y\in(\N_0)^n$, we say that $x\leq y$ if $x_i\leq y_i$ for all $i\in[n]$. A set $S\subset(\N_0)^n$ is called \emph{monotone decreasing (respectively, increasing)} if $x\in S$ implies that $y\in S$ for all $y\leq x$ (respectively, $x\leq y$).

\begin{lem}[\cite{MU05}, Theorem 5.10]\label{lem:poisson-bound}
Let $\{X_m\}_{m\in[n]}$ be the number of balls in bins $m\in[n]$ when $\lfloor\theta n\rfloor$ balls are independently and uniformly placed into $n$ bins. Let $\{Y_m\}_{m\in[n]}$ be independent Poisson random variables with parameter $\theta$. For any monotone set $S\subseteq[n]$, we have
$$
\P((X_1, \cdots, X_n)\in S)\leq 2\P((Y_1, \cdots, Y_n)\in S).
$$
\end{lem}

Two additional lemmata from \cite{FGG18} provide concentration bounds for the maximum load over a subset of bins and on  the number of bins with load above a certain level in the balls-and-bins process, respectively.

\begin{lem}[\cite{FGG18}, Lemma 2.2] \label{lem:max-bound}
Let $\{X_m\}_{m\in[n]}$ be the number of balls in bins $m\in[n]$ when  $\lfloor\theta n\rfloor$ balls are independently and uniformly placed into $n$ bins. For $k\in\lfloor\theta n\rfloor$ and $S\subseteq [n]$, we have
$$
\P\left(\max_{m\in S}X_m<k\right)\leq 2\exp\left(-\frac{\theta^k|S|}{ek!}\right).
$$
\end{lem}

\begin{lem}[\cite{FGG18}, Lemma 2.3]\label{lem:nonempty-bound}
Let $\{X_m\}_{m\in[n]}$ be the number of balls in bins $m\in[n]$ when  $\lfloor\theta n\rfloor$  balls are independently and uniformly placed into $n$ bins. For any $S\subseteq [n]$, we have
$$
\P\left(|\{m\in S: X_m>0\}|\leq \frac{\theta |S|}{2e}\right)\leq 2\exp\left(-\frac{\theta^2|S|}{2e^2}\right).
$$
\end{lem}

\section{MaxLoad upper bound: \texorpdfstring{$\left(\frac{d\log n}{\log\log n}\right)^{1/d}$}{(d lg n / log log n)(1/d)}-threshold strategy}

This section is dedicated to establishing a tight upper bound for the maximum load
of the $\ell$-threshold strategy with $\ell=\left(\frac{d\log n}{\log\log n}\right)^{1/d}$, from which the upper bound in Theorem \ref{thm:main} is an immediate consequence. Our proof essentially follows from iterations of the argument in \cite{FGG18}. Observe that with this choice of $\ell$, we have
$$
\ell^d\log \ell= \left(1-\frac{\log\log\log n - \log d}{\log\log n}\right)\log n,
$$
which implies
\begin{equation}\label{eq:ell n relations}
\ell^{\ell^d}=n^{1-o(1)}.
\end{equation}

\begin{prop}\label{prop:upper bound}
Let $\ell=\Big(\frac{d\log n}{\log\log n}\Big)^{1/d}$. For any $\epsilon>0$, there exists $n_0=n_0(d, \rho, \epsilon)$ such that for all $n>n_0$, the $\ell$-threshold strategy $f$ satisfies
\begin{align}\label{eq:upper bound}
\P\left(\maxload^f(\lfloor\rho n\rfloor)>(d+\epsilon)\ell\right)< n^{-\epsilon/3}.
\end{align}
\end{prop}

\begin{proof}
By definition, the $\ell$-threshold strategy guarantees that $L_{i, m}^f(r)\leq \ell$ for all $i\in[d-1]$, $m\in [n]$ and $r\in\N$.
%Notice that for each bin $m\in[n]$ the overall load $L_m^f(\rho n)$ is the accumulation of loads from $d$ rounds allocation, i.e., $L_m^f(\rho n)=\sum_{i=1}^dL_{i, m}^f(r_{i-1})$.
Set $L=(1+\epsilon)\ell$. We have
\begin{align}\label{eq:upper-maxload} % OHAD: Why not equality?
\P\left(\text{MaxLoad}^f(\lfloor\rho n\rfloor)>(d+\epsilon)\ell\right) \leq \P\left(\max_{m\in[n]}L_{d, m}^f(r_{d}(\lfloor\rho n\rfloor))>L\right).
\end{align}
We will upper bound this probability conditioned on the event that $r_i:=r_i(\lfloor\rho n\rfloor)$  is bounded above by $\beta_i$ for all $i\in[d]$, where the sequence $\{\beta_i\}_{i\in[d]}$ is chosen along the proof so that this event occurs with high probability. Using the law of total probability, we have
\begin{align}\label{eq:cond}
\P\left(\max_{m\in[n]}L_{d, m}^f(r_{d})>L\right) &\leq \P\left(\max_{m\in[n]}L_{d, m}^f(r_{d})>L\ \big|\ r_{d}\leq\beta_{d}\right)+\P(r_{d}>\beta_{d})\nonumber\\
&\leq\P\left(\max_{m\in[n]}L_{d, m}^f(\lfloor\beta_{d}\rfloor)>L\right)+\P(r_{d}>\beta_{d})\nonumber\\
&\leq\P\left(\max_{m\in[n]}L_{d, m}^f(\lfloor\beta_{d}\rfloor)>L\right)+\sum_{i=2}^{d}\P(r_i>\beta_i\ |\ r_{i-1}\leq\beta_{i-1}).
\end{align}
Next, we define the sequence $\{\beta_i\}_{i\in[d]}$ in an iterative way such that the 2nd term of \eqref{eq:cond} is small. We set $r_1=\lfloor\rho n\rfloor$ and $\beta_1=\rho n$. To define $\beta_2$ and upper bound $\P(r_2>\beta_2\ |\ r_1\leq \beta_1)$, we introduce independent  Poissons $\{Y_{k}\}_{k\in [n]}$ with parameter $\rho_1=\beta_1/n$. Define
$$
Y=\sum_{k=1}^n\max\{Y_{k}-\ell, 0\}.
$$
By Lemma \ref{lem:poisson-bound}, we have
$$
\P(r_2>\beta_2)\leq 2\P(Y>\beta_2).
$$
Since
\begin{align*}
\E \left(e^{\max\{Y_{1}-\ell, 0\}}\right)
\leq 1+ \sum_{k=\ell+1}^\infty e^{k-\ell}\cdot e^{-\rho_1}\frac{\rho_1^k}{k!}<1+\frac{\rho_1^\ell}{\ell!}<\exp\left(\frac{\rho_1^\ell}{\ell!}\right),
% &=1+e^{-\rho-l}\sum_{k=l+1}^\infty\frac{(\rho e)^k}{k!}\\
% &\leq 1+e^{-\rho-l}\cdot e^{\rho e}\cdot\frac{(\rho e)^{l+1}}{(l+1)!}\\
% &=1+\frac{\rho e^{\rho e-\rho+1}}{l+1}\cdot\frac{\rho^l}{l!}\\
\end{align*}
we have
$$
\P(Y>\beta_2)=\P\left(e^{Y}>e^{\beta_2}\right)< \exp\left(\frac{n\rho_1^\ell}{\ell!}-\beta_2\right).
$$
To make the right-hand side of the above inequality small, we set $\beta_2=2n\rho_1^\ell/\ell!$ and obtain
$$
\P(r_2>\beta_2)< 2\exp\left(-\frac{\beta_2}{2}\right).
$$
Once $\beta_{i-1}$ has been defined, we define $\beta_i$ and upper bound $\P(r_i>\beta_i\ |\ r_{i-1}\leq \beta_{i-1})$ by repeating the above argument, replacing $\{Y_{k}\}_{k\in [n]}$ with independent Poisson random variables with parameter $\rho_{i-1}:=\beta_{i-1}/n$. Applying Lemma \ref{lem:poisson-bound} to such variables we obtain
$$
\P(r_i>\beta_i\ |\ r_{i-1}\leq \beta_{i-1})\leq 2\P(Y>\beta_i)<2\exp\left(\frac{n\rho_{i-1}^\ell}{\ell!}-\beta_i\right).
$$
Setting $\beta_i=2n\rho_{i-1}^\ell/\ell!$, this yields
$$
\P(r_i>\beta_i\ |\ r_{i-1}\leq \beta_{i-1})< 2\exp\left(-\frac{\beta_i}{2}\right).
$$
%$$
% \rho_{d-1}=\rho^{l^{d-2}}\Big(\frac{2}{l!}\Big)^{\frac{l^{d-2}-1}{l-1}}
% $$
% and
Putting together our definition of $\beta_i$ and $\rho_i$, we obtain the recursive formula
$$\frac{\beta_i}{n}=\frac{2}{\ell!}\cdot\left(\frac{\beta_{i-1}}{n}\right)^\ell,$$
with $\beta_1=\rho n$. By solving this recursion we obtain
\begin{align}\label{eq:beta-d}
\beta_{d}=n \rho^{\ell^{d-1}}\left(\frac{2}{\ell!}\right)^{\sum_{i=0}^{d-2} \ell^i}
= n\rho^{\ell^{d-1}}\left(\frac{2}{\ell!}\right)^{\frac{\ell^{d-1}-1}{\ell-1}}.
\end{align}
Observe that $2/\ell!>\ell^{-(\ell-1)}$ when $\ell\geq3$ (using the fact that all $k\in\N$ satisfy $k!\leq e\sqrt{k}(k/e)^k$). Thus for $n\ge 3$ we have
$$
\beta_{d}> n\left(\frac{\rho}{\ell}\right)^{\ell^{d-1}}>n\ell^{-2\ell^{d-1}}>n\ell^{-\ell^{d}}=n^{1-o(1)},
$$
where the second and third inequalities hold when $\ell>\max\{2, 1/\rho\}$, and the last identity uses the observation $\ell^{\ell^d}=n^{1-o(1)}$ as stated in \eqref{eq:ell n relations}. Since $\{\beta_i\}_{i\in[d]}$ is a decreasing sequence, we deduce the following bound on the
second term of \eqref{eq:cond}.
\begin{align}\label{eq:cond-prob}
\P(r_i>\beta_i\ |\ r_{i-1}\leq \beta_{i-1})< 2\exp\left(-\frac{\beta_{d}}{2}\right)< e^{-n^{1-o(1)}}.
\end{align}
To estimate the first term of \eqref{eq:cond}, we consider $\{Y_{k}\}_{i\in[n]}$, a collection of independent Poisson random variables with parameter $\rho_d=\beta_{d}/n$ and apply Lemma \ref{lem:poisson-bound} to obtain
$$
\P\left(\max_{m\in[n]}L_{d, m}^f(\lfloor\beta_{d}\rfloor)>L\right)\leq 2\P\left(\max_{1\leq k\leq n}Y_{k}\geq L\right).
$$
Notice that
$$
\P(Y_{ 1}\geq L)=e^{-\rho_d}\sum_{k=L}^\infty\frac{\rho_d^k}{k!}\leq \frac{\rho_d^L}{L!}.
$$
Take a union bound to obtain
$$
\P\left(\max_{m\in[n]}L_{d, m}^f(\lfloor\beta_{d}\rfloor)>L\right)\leq 2n\frac{\rho_d^L}{L!}=\frac{2n}{L!}\left(\rho^{\ell^{d-1}}\left(\frac{2}{\ell !}\right)^{\frac{\ell^{d-1}-1}{\ell-1}}\right)^{L},
$$
where the identity follows from our definition $\rho_d=\beta_d/n$ and the formular \eqref{eq:beta-d}. Using the inequality $k!\geq \sqrt{2\pi k}(k/e)^k$ for any $k\in\N$, we have $2/L!<(e/L)^{L}<1$ when $\ell>e$ and
\begin{align*}
\left(\frac{2}{\ell!}\right)^{L\cdot\frac{\ell^{d-1}-1}{\ell-1}}<\left(\frac{e}{\ell}\right)^{\ell L\cdot\frac{\ell^{d-1}-1}{\ell-1}}=\left(\frac{e}{\ell}\right)^{(1+\epsilon)\ell^2\cdot\frac{\ell^{d-1}-1}{\ell-1}}<\left(\frac{e}{\ell}\right)^{(1+\epsilon)\ell^d}.%<\ell^{-(1+\epsilon/3)\ell^d}.
\end{align*}
Therefore, we end up with
\begin{align}\label{eq:last-round-prob}
\P\left(\max_{m\in[n]}L_{d, m}^f(\lfloor\beta_{d}\rfloor)>L\right) &< n\left(\frac{\rho e}{\ell}\right)^{(1+\epsilon)\ell^d}\notag\\%&\leq n \rho^{(1+\epsilon)\ell^d}\cdot \ell^{-(1+\epsilon/3)\ell^d} \nonumber\\
&< n\ell^{-(1+\epsilon/2)\ell^d}\nonumber\\
&=\exp\left(\log n-\log n\left(1+\frac{\epsilon}{2}\right)\left(1-\frac{\log\log\log n - \log d}{\log\log n}\right)\right)\nonumber\\
&<n^{-\epsilon/3}.
\end{align}
The second inequality holds automatically whenever $\rho e<1$. Otherwise, it holds for all $n$ large enough to satisfy $\epsilon>\log(\rho e)/(\log \sqrt{\ell}-\log(\rho e))=\Theta((\log n-\log\log n)^{-1})$. The last inequality holds  for all $n$ large enough to satisfy $\epsilon>7(\log\log\log n - \log d)/\log\log n$. The desired statement \eqref{eq:upper bound} now follows from \eqref{eq:upper-maxload}, \eqref{eq:cond}, \eqref{eq:cond-prob} and \eqref{eq:last-round-prob}.
\end{proof}

\begin{rmk}
From the above proof, we see that the parameter $\epsilon$ is allowed to depend on $n$ and Proposition \ref{prop:upper bound} remains valid as long as $\epsilon>C\log\log\log n/\log\log n$ for an appropriate absolute constant $C$. Thus, with this choice of $\epsilon$, the upper bound in Theorem \ref{thm:main} is an immediate consequence of Proposition \ref{prop:upper bound}.
\end{rmk}

\section{MaxLoad lower bound: generic thinning strategies}\label{sec: lower}

%Throughout this section, we replace the superscript $f$ by $d$ to indicate that the statement holds for any $d$-thinning strategy. Our main result is the following lower bound for the maximum load of any $d$-thinning strategy.

%Our proof is essentially a multi-round iteration of the argument used in Proposition 5.1 of \cite{FG17}. The goal in each round is to show that the probability of the maximum load over a subset of bins below certain level is small. This is reduced to a similar allocation problem in the next round using thinning strategies with one less retry. This `self-similar' feature enables us to establish an inductive bound.

In this section, we prove the following proposition, from which the lower bound in Theorem \ref{thm:main} is an immediate consequence.

\begin{prop}\label{prop:lower bound}
For any $\epsilon>0$, there exists $n_0=n_0(d, \rho, \epsilon)$ such that for all $n>n_0$, any $d$-thinning strategy $f$ satisfies
\begin{align}\label{eq:lower bound}
\P\left(\maxload^f (\lfloor\rho n\rfloor)<(d-\epsilon)\ell\right)<\exp\left(-n^{\epsilon/3d}\right).
\end{align}
\end{prop}

We first prove a reduction lemma which reduces a d-thinning allocation problem to a
$(d-1)$-thinning problem. The above statement will follow from iterations of this thinning
reduction lemma.

Let us consider the following more general problem of allocating $\lfloor\gamma\rfloor$ balls into $n$ bins using a $d$-thinning strategy $f$. Let $s>0$ and let $S\subseteq [n]$ be a subset. We denote by $A^f$ the event that the maximum load over $S$ is less than $s$, i.e.,
\begin{align}\label{eq:A}
A^f=\left\{\text{MaxLoad}_{S}^{f}(\lfloor\gamma\rfloor)<s\right\}.
\end{align}
Our goal is to upper bound the probability $\P\left(A^f\right)$. We divide the allocation process into $\lceil s\rceil$ stages, each of which consists of allocating  $w=\lfloor\gamma/s\rfloor$ balls except for the final stage in which the remaining balls are allocated.  We define the level sets $\big\{S_k^f\big\}_{k=0, 1, \cdots, \lceil s\rceil}$  as
\begin{align}\label{eq:S_i+1}
S^f_k=\left\{m\in S: L_m^f(kw)\geq k\right\},
\end{align}
where $L_m^{f}(t)$, defined as per \eqref{eq:load-m}, represents the load of bin $m$ after $t$ balls have been allocated using a $d$-thinning strategy $f$.  Given $\{\gamma_k\}_{k=0, 1, \cdots, \lceil s\rceil}$, an arbitrary sequence with $\gamma_0=|S|$, we
introduce the following events
\begin{equation}\label{eq:E_i+1}
\begin{split}E^f_k&=\left\{\big|S^f_k\big|<\gamma_k\right\},\\
F^f_k&=\left\{\text{MaxLoad}_{S}^f(kw)<s\right\}.
\end{split}
\end{equation}
 Here, $E_k^f$ is the event that after $k$ stages there are less than $\gamma_k$ bins containing $k$ or more balls, and $F_k^f$ is the event that after $k$ stages the maximum load over $S$ is less than $s$. Recall that  $f^+$, defined as per \eqref{eq:fplus def}, is the $(d-1)$-thinning strategy induced from $f$. Our reduction argument utilizes the following events
%\begin{align*}%\label{eq:Ak}
%A_k=\{\text{MaxLoad}_{S_{k-1}}^{d-1}(\lfloor\gamma_k\rfloor)<s-(k-1)\},
%\end{align*}
%\begin{align*}%\label{eq:E_i+1'}
%B_k=\{\psi_{S_{k-1}}^d(w)< 2\gamma_k\},
%\end{align*}
\begin{equation}\label{eq:Ak}
\begin{split}
A^{f^+}_k&=\Big\{\maxload_{S_{k-1}^f}^{f^+}(\lfloor\gamma_k\rfloor)<s-(k-1)\Big\},\\
B^f_k&=\Big\{\psi_{S^f_{k-1}}^f(w)< 2\gamma_k\Big\},
\end{split}
\end{equation}
where $\psi_{S_{k-1}^f}^f(w)$ was defined in \eqref{eq:primary-super-level} as the number of bins in $S_{k-1}^f$ selected as primary slots in the allocation of $w$ balls using the $d$-thinning strategy $f$. We require the following reduction lemma.

\begin{lem}\label{lem:reduction}
Let $f$ be a $d$-thinning strategy. For any $\gamma, s>0, S\subseteq [n]$ and arbitrary sequence  $\{\gamma_k\}_{k=0, 1, \cdots, \lceil s\rceil}$ with $\gamma_0=|S|$, under the notations above, we have
% Let $d\in\N$ and let $f$ be a $d$-thinning strategy.  Let $\gamma, s>0, S\subseteq [n]$, and let $\{\gamma_k\}_{k=0, 1, \cdots, \lceil s\rceil}$ be  an arbitrary sequence with $\gamma_0=|S|$. We define
% \begin{equation}\label{eq:Ak}
% \begin{split}
% A^{f^+}_k&=\Big\{\maxload_{S_{k-1}^f}^{f^+}(\lfloor\gamma_k\rfloor)<s-(k-1)\Big\},\\
% B^f_k&=\Big\{\psi_{S^f_{k-1}}^f(w)< 2\gamma_k\Big\},
% \end{split}
% \end{equation}
% where  $f^+$, defined in~\eqref{eq:fplus def}, is the induced $(d-1)$-thinning strategy. Then we have
\begin{align}\label{eq:reduction}
\P\left(A^f\right)\leq \sum_{k=1}^{\lceil s\rceil}\P\left(\left(A_k^{f^+}\cup B_k^f\right)\cap \left(E_{k-1}^f\right)^c\right).
\end{align}
\end{lem}

\begin{proof}
For simplicity, we write $A=A^f, S_k=S^f_k, E_k=E_k^f, F_k=F_k^f, A_k=A_k^{f^+}$ and $B_k=B^f_k$. %and $\phi_{S_{k-1}}=\phi_{S_{k-1}}^f$.
Clearly, $A\subseteq E_{\lceil s\rceil}\cap F_{\lceil s\rceil}$ and $F_k\subseteq F_{k-1}$. From the law of total probability, we obtain
$$
\P(E_k\cap F_k)\leq \P(E_{k-1}\cap F_{k-1})+\P((E_k\cap F_k)\cap E_{k-1}^c).
$$
By induction, we thus have,
\begin{align}\label{eq:induction}
\P(A)\leq \P\left(E_{\lceil s\rceil}\cap F_{\lceil s\rceil}\right)\leq \sum_{k=1}^{\lceil s\rceil}\P(E_k\cap F_k\cap E_{k-1}^c),
\end{align}
where $E_0=\emptyset$ as we set $\gamma_0=|S|$. Notice that the size of $S_k$ is at least the number of bins in $S_{k-1}$ which  receive at least one ball in the $k$-th stage. Hence, we have
\begin{align}\label{eq:inclusion-e}
E_k\subseteq \left\{\phi^f_{S_{k-1}}(w)< \gamma_k\right\},
\end{align}
where $\phi_{S_{k-1}}^f(w)$, defined in \eqref{eq:super-level}, is the number of non-empty bins in $S_{k-1}$ after allocating $w$ balls using our $d$-thinning strategy $f$. %Similarly, the maximum load over $S$ after $k$ stages is at least the maximum load over $S_{k-1}$ in the  $k$-th stage plus $k-1$. Hence, we have
Using the inequality
$$
\text{MaxLoad}_S^f(kw)\geq \text{MaxLoad}_{S_{k-1}}^f((k-1)w)+\text{MaxLoad}_{S_{k-1}}^f(w)
$$
and the fact that $ \text{MaxLoad}_{S_{k-1}}^f((k-1)w)\geq k-1$, we have
\begin{align}\label{eq:inclusion-f}
F_k\subseteq \left\{\text{MaxLoad}_{S_{k-1}}^f(w)<s-(k-1)\right\}.
\end{align}
From the definitions of $r_2$ and $f^+$  in \eqref{eq:r_i} and \eqref{eq:fplus def}, respectively, it follows that
\begin{equation}\label{eq:inclusion-a}
\begin{split}
A_k^c\cap\{r_2(w)&> \gamma_k\}\subseteq\left\{\text{MaxLoad}_{S_{k-1}}^f(w)\geq s-(k-1)\right\},\\
B_k^c\cap\{r_2(w)&\leq \gamma_k\}\subseteq  \left\{\phi^f_{S_{k-1}}(w)\geq \gamma_k\right\}.
\end{split}
\end{equation}
Putting together  \eqref{eq:inclusion-e}, \eqref{eq:inclusion-f} and \eqref{eq:inclusion-a}, we have
\begin{align*}%\label{eq:inclusion}
E_k\cap F_k\subseteq A_k\cup B_k.
\end{align*}
This, together with \eqref{eq:induction}, proves the reduction lemma.

\end{proof}

\begin{proof}[Proof of Proposition \ref{prop:lower bound}]
The proof follows from a multi-round iteration of Lemma \ref{lem:reduction}. Recall that $f^j$ with $j\in[d]$, given in \eqref{fj def}, is the induced $(d-j+1)$-thinning strategy. Particularly, we have $f^1=f$. We start with $\gamma=\rho n$, $S:=[n]$, and $s:=s_1=(d-\epsilon)\ell$. The goal is to upper bound $\P(A)$, where $A:=A^{f^1}$ is the event defined as per \eqref{eq:A}. Set $w:=\lfloor w_1\rfloor$ with $w_1=\gamma/s_1$.  Let  $k_1\in\{0, \cdots, \lceil s_1\rceil\}$.
Recall our definitions of $S_{k_1}^{f^1}$ and $E^{f^1}_{k_1}, F^{f^1}_{k_1}$ in \eqref{eq:S_i+1} and
\eqref{eq:E_i+1}, respectively.  For simplicity, we write
$S_{1,k_1}=S_{k_1}^{f^1}, E_{1,k_1}=E^{f^1}_{k_1}$ and $F_{1,k_1}=F^{f^1}_{k_1}$.
The sequence $\{\gamma_{1, k_1}\}_{k_1\in\{0, \cdots, \lceil s_1\rceil\}}$ is defined as follows
\begin{align*}%\label{eq:gamma-1}
\gamma_{1, k_1}=n\left(\frac{\theta_1}{4e}\right)^{k_1},
\end{align*}
where $\theta_1=w_1/n$. We denote  $A_{1, k_1}=A_{k_1}^{f^2}$ and $B_{1,k_1}=B_{k_1}^{f^1}$, which are defined as per \eqref{eq:Ak}, using the fact that $f^+=f^2$.  In these notations, the first subscript (also used as the subscript of $k$)
indicates the first round of allocation (i.e., primary allocation). We apply Lemma \ref{lem:reduction} with aforementioned parameters $\gamma, S, s$ and $\{\gamma_{1, k_1}\}_{k_1\in\{0, \cdots, \lceil s_1\rceil\}}$ to obtain
\begin{align}\label{eq:inductive bound'}
\P(A)\leq \sum_{k_1=1}^{\lceil s_1\rceil}\P((A_{1, k_1}\cup B_{1, k_1})\cap E_{1, k_1-1}^c).
\end{align}
Lemma \ref{lem:nonempty-bound} with $\theta=\theta_1$ and the set there $S=S_{1, k_1-1}$ yields
\begin{align*}%\label{eq:prob-E_{i+1}'}
\P(B_{1, k_1}\cap E_{1, k_1-1}^c)\leq 2\exp\left(-\frac{\theta_1^2|S_{1,k_1-1}|}{2e^2}\right)\leq 2\exp\left(-\frac{\theta_1^2\gamma_{1, k_1-1}}{2e^2}\right),
\end{align*}
where the second inequality uses the fact that, under $E_{1, k_1-1}^c$, we have $|S_{1, k_1-1}|\geq \gamma_{1, k_1-1}$. This, together with \eqref{eq:inductive bound'}, yields
\begin{align}\label{eq:inductive bound}
\P(A)%&\leq  \sum_{k_1=1}^{\lceil s_1\rceil}\P(A_{1, k_1}\cap E_{1, k_1-1}^c) + \sum_{k_1=1}^{\lceil s_1\rceil}\P(B_{1, k_1}\cap E_{1, k_1-1}^c)\notag\\
&\le \sum_{k_1=1}^{\lceil s_1\rceil}\P(A_{1, k_1}\cap E_{1, k_1-1}^c)+2\sum_{k_1=1}^{\lceil s_1\rceil} \exp\left(-\frac{\theta_1^2\gamma_{1, k_1-1}}{2e^2}\right).
%\P\left(\text{MaxLoad}_{S_{i-1, k_{i-1}}}^{d-i+1}(\gamma_{i-1, k_{i-1}+1})<s_i\big|E_{i-1, k_{i-1}}^c\right)
\end{align}

Our goal is therefore to upper bound $\P(A_{1, k_1} \cap E_{1, k_1-1}^c)$. This will be achieved by establishing an upper bound for $\P(A_{1, k_1})$ analogous to \eqref{eq:inductive bound}, which involves the $(d-2)$-thinning strategy $f^3$ instead of the $(d-1)$-thinning strategy $f^2$. We continue this thinning reduction procedure for $(d-1)$ rounds until the problem is reduced to the analysis of the $1$-thinning strategy $f^d$, which is a problem of deterministic computation.

Firstly, we inductively define some notations analogous to those we used to obtain \eqref{eq:inductive bound}. For $i\in[d]$, we define
\begin{equation}\label{eq:def si}
\begin{split}
s_{i}&=s_{i-1}-(k_{i-1}-1), \\%=(d-\epsilon)l-\sum_{j=0}^{i-1}(k_j-1),
w_{i}&=\gamma_{i-1, k_{i-1}}/s_{i},
\end{split}
\end{equation}
%(In order to keep consistency with the first round, we formally set $k_0=1$, $\gamma_{0, 1}=\rho n$, $S_{0, 0}=[n]$, $A_{0, 1}=A$, and $E_{0, 0}=\emptyset$).
where  $k_{i-1}\in\{0, 1, \cdots, \lceil s_{i-1}\rceil\}$, and the sequence $\{\gamma_{i, k_i}\}_{k_i\in\{0, \cdots, \lceil s_{i}\rceil\}}$  is defined as follows
\begin{align}\label{eq:gamma-i-iteration}
\gamma_{i, k_i}=\gamma_{i-1, k_{i-1}}\left(\frac{\theta_i}{4e}\right)^{k_i},
\end{align}
where $\theta_i=w_i/n$.
For $i\in[d-1]$, we write $S_{i,k_i}=S_{k_i}^{f_{i}}$, and $E_{i,k_i}=E^{f^{i}}_{k_i}, F_{i,k_i}=F^{f^{i}}_{k_i}$, and  $A_{i,k_i}=A_{k_i}^{f^{i+1}}, B_{i,k_i}=B_{k_i}^{f^{i}}$, which are defined in \eqref {eq:S_i+1}, \eqref{eq:E_i+1} and \eqref{eq:Ak}, respectively, with $f:=f^{i}$ and $S:=S_{i-1, k_{i-1}-1}$. %Here the event $A_{i, k_{i}}=A_{k_{i}}^{f^{i+1}}$ is defined as per \eqref{eq:A} with respect to $f^i$. That is,
Here, $A_{i,k_i}$ is the event that the maximum load over $S_{i, k_{i}-1}=S_{k_{i}-1}^{f^{i}}$ of allocating $\lfloor\gamma_{i, k_{i}}\rfloor$ balls into $n$ bins using the $(d-i)$-thinning strategy $f^{i+1}$ is below $s_{i+1}$. To keep consistency, we formally set $s_0=(d-\epsilon)\ell$, $k_0=1$, $\gamma_{0, 1}=\rho n$, $S_{0, 0}=[n]$ and $A_{0, 1}=A$. %and $E_{0, 0}=\emptyset$.

We state the general step of the recursion as follows. In the $i$-th step, our task is to upper bound $\P(A_{i-1, k_{i-1}})$.
%Suppose that we have already defined  $\gamma:=\lfloor\gamma_{i-1, k_{i-1}}\rfloor$, $S:=S_{i-1,k_{i-1}}$, $s:=s_i=(d-\epsilon)l-\sum_{j=0}^{i-1}(k_j-1)$, and a $(d-i+1)$-thinning strategy $f_{i-1}$. The goal is to upper bound $\P(A_{i-1, k_{i-1}})$, where $A_{i-1, k}=A_{k}^{f_{i-1}}$ is defined in \eqref{eq:Ak} with $f=f_{i-2}$ and $g=f_{i-1}$.
Analogous to \eqref{eq:inductive bound'}, we apply Lemma \ref{lem:reduction} with the parameters $\gamma=\lfloor\gamma_{i-1, k_{i-1}}\rfloor, S=S_{i-1, k_{i-1}-1}, s=s_{i}$ and $\{\gamma_{i, k_i}\}_{k_i\in\{0, \cdots, \lceil s_{i}\rceil\}}$ to obtain
\begin{align*}%\label{eq:i-inductive bound'}
\P(A_{i-1, k_{i-1}}) \leq \sum_{k_i=1}^{\lceil s_i\rceil}\P((A_{i, k_i}\cup B_{i, k_i})\cap E_{i, k_i-1}^c).
%\P\left(\text{MaxLoad}_{S_{i-1, k_{i-1}}}^{d-i+1}(\gamma_{i-1, k_{i-1}+1})<s_i\big|E_{i-1, k_{i-1}}^c\right)
\end{align*}
Lemma \ref{lem:nonempty-bound} with $\theta=\theta_i$ and the set there $S=S_{i, k_i-1}$ yields
\begin{align*}%\label{eq:prob-E_{i+1}'}
\P(B_{i, k_i}\cap E_{i, k_i-1}^c)\leq  2\exp\left(-\frac{\theta_{i}^2|S_{i,k_i-1}|}{2e^2}\right) \le 2\exp\left(-\frac{\theta_{i}^2\gamma_{i, k_i-1}}{2e^2}\right),
\end{align*}
where the second inequality again uses the fact that, under $E_{i, k_i-1}^c$, we have $|S_{i, k_i-1}|\geq \gamma_{i, k_i-1}$. Analogous to \eqref{eq:inductive bound}, we have
\begin{align}\label{eq:i-inductive bound}
\P(A_{i-1, k_{i-1}})\leq \sum_{k_i=1}^{\lceil s_i\rceil}\P(A_{i, k_i}\cap E_{i, k_i-1}^c)+2\sum_{k_i=1}^{\lceil s_i\rceil} \exp\left(-\frac{\theta_{i}^2\gamma_{i, k_i-1}}{2e^2}\right).
\end{align}
Putting together \eqref{eq:inductive bound} and \eqref{eq:i-inductive bound}, we obtain
\begin{align}\label{eq:final-inductive bound}
\P(A)\leq \sum_{k_1=1}^{\lceil s_1\rceil}\cdots\sum_{k_{d-1}=1}^{\lceil s_{d-1}\rceil}\P(A_{d-1, k_{d-1}}\cap E_{d-1, k_{d-1}-1}^c) +2\sum_{i=1}^{d-1}\sum_{k_1=1}^{\lceil s_1\rceil}\cdots\sum_{k_i=1}^{\lceil s_i\rceil}\exp\left(-\frac{\theta_i^2\gamma_{i, k_i-1}}{2e^2}\right).
\end{align}

We conclude the proof with the estimate of the right-hand side of \eqref{eq:final-inductive bound}. Firstly, we establish lower bounds for $\theta_i$ and $\gamma_{i,k_i}$.
Recall that $\gamma_{0, 1}=\rho n$, $\theta_i=w_i/n$,  $w_i=\gamma_{i-1, k_{i-1}}/s_i$ and $s_i=(d-\epsilon)\ell-\sum_{j=0}^{i-1}(k_j-1)$. For $i\in[d-1]$, we iterate equations \eqref{eq:def si} and \eqref{eq:gamma-i-iteration} to obtain
\begin{align}
\gamma_{i, k_i}&=n\prod_{j=1}^{i}\left(\frac{\theta_j}{4e}\right)^{k_j}, \label{eq:gamma-i}\\
\theta_{i+1}&=\frac{1}{s_{i+1}}\prod_{j=1}^{i}\left(\frac{\theta_j}{4e}\right)^{k_j} \label{eq:theta-i}.
\end{align}
Denote $\tilde{\rho}=\min\{1, \rho\}$. Towards obtaining a lower bound
for $\theta_i$, we iterate \eqref{eq:theta-i} to obtain
\begin{align*}
\frac{\theta_1}{4e}&=\frac{\rho}{4e(d-\epsilon)\ell}>\frac{\tilde{\rho}}{4ed\ell},\\
\frac{\theta_2}{4e}&>\frac{1}{4e(d-\epsilon)\ell}\left(\frac{\rho}{4e(d-\epsilon)\ell}\right)^{k_1}>\left(\frac{\tilde{\rho}}{4ed\ell}\right)^{k_1+1},
\end{align*}
and repeat this process with the simple fact $s_{i+1}<(d-\epsilon)\ell$ (see \eqref{eq:def si}) to inductively verify that for $i\geq 2$ we have
$$\frac{\theta_{i+1}}{4e}>\left(\frac{1}{4e(d-\epsilon)\ell}\right)^{\prod_{j=2}^{i}(k_j+1)} \left(\frac{\rho}{4e(d-\epsilon)\ell}\right)^{k_1\prod_{j=2}^{i}(k_j+1)}>\left(\frac{\tilde{\rho}}{4ed\ell}\right)^{\prod_{j=1}^{i}(k_j+1)}.$$
Hence, for any $i\in[d-1]$, we have
\begin{align}
\frac{\theta_{i+1}}{4e}&>\left(\frac{\tilde{\rho}}{4ed\ell}\right)^{\prod_{j=1}^{i}(k_j+1)}, \label{eq:theta-bound}\\
\gamma_{i, k_i}&>n\left(\frac{\tilde{\rho}}{4ed\ell}\right)^{\prod_{j=1}^{i}(k_j+2)}, \label{eq:gamma-bound}
\end{align}
where \eqref{eq:gamma-bound} follows from plugging \eqref{eq:theta-bound} into \eqref{eq:gamma-i}.

Next, we upper bound each term of the right-hand side of \eqref{eq:final-inductive bound}.
We begin with the easier 2nd term.

\textbf{The 2nd term of \eqref{eq:final-inductive bound}}.
Notice that $\theta_i$ and $\gamma_{i, k_i}$ are decreasing in both $i$ and $k_i$.
Using the lower bounds \eqref{eq:theta-bound} and \eqref{eq:gamma-bound}, we have
\begin{align}
\sum_{i=1}^{d-1}\sum_{k_1=1}^{\lceil s_1\rceil}\cdots\sum_{k_i=1}^{\lceil s_i\rceil}\exp\left(-\frac{\theta_i^2\gamma_{i, k_i-1}}{2e^2}\right)
&< d(dl)^{d-1}\exp\left(-\frac{\theta_{d-1}^2\gamma_{d-1, k_{d-1}-1}}{2e^2}\right)\notag\\
&< d(d\ell)^{d-1}\exp\left(-8n\left(\frac{\tilde{\rho}}{4ed\ell}\right)^{3\prod_{j=1}^{d-1}(k_j+2)}\right)\notag\\
&< d(d\ell)^{d-1}\exp\left(-8n\left(\frac{\tilde{\rho}}{4ed\ell}\right)^{3\left(\frac{d-\epsilon}{d-1}\ell+3\right)^{d-1}}\right)\notag\\
&< d(d\ell)^{d-1}\exp\left(-8n\left(\frac{\tilde{\rho}}{4ed\ell}\right)^{3(2\ell+3)^{d-1}}\right)\notag\\
&=e^{-n^{1-o(1)}}\label{eq: term2}.
\end{align}
In the 3rd inequality, we use the inequality of arithmetic and geometric means, and the fact that $\sum_{j=1}^{d-1}(k_j+2)\leq (d-\epsilon)\ell+3(d-1)$, which  follows from $s_d=(d-\epsilon)\ell-\sum_{j=0}^{d-1}(k_j-1)>0$. Observe that the quantity inside the exponential function is of order $n\ell^{-\ell^{d-1}}$ and the fact that $n\ge  \ell^{\ell^d}$ as stated in \eqref{eq:ell n relations}. Hence, the last identity holds for sufficiently large $n$ depending on $d$ and $\rho$ but not $\epsilon$.

\textbf{The 1st term of \eqref{eq:final-inductive bound}}. Notice that $A_{d-1, k_{d-1}}$ is the event that the maximum load over $S_{d-1, k_{d-1}-1}$ after throwing $\lfloor\gamma_{d-1, k_{d-1}}\rfloor$ balls into $n$ bins without retry (i.e., 1-thinning) is less than $s_d$. Apply Lemma \ref{lem:max-bound} with $\theta=\gamma_{d-1, k_{d-1}}/n$ and $S=S_{d-1, k_{d-1}-1}$ to obtain
$$
\P(A_{d-1, k_{d-1}}\cap E_{d-1, k_{d-1}-1}^c)\leq 2\exp\left(-\left(\frac{\gamma_{d-1, k_{d-1}}}{n}\right)^{s_d}\frac{\gamma_{d-1, k_{d-1}-1}}{e s_d!}\right),
$$
where we use the fact that, under $E_{d-1, k_{d-1}-1}^c$, we have $|S_{d-1, k_{d-1}-1}|\geq \gamma_{d-1, k_{d-1}-1}$.  Using \eqref{eq:gamma-bound} and the fact that $\gamma_{d-1, k_{d-1}}<\gamma_{d-1, k_{d-1}-1}$, we have
\begin{equation}\label{eq:last step bound}
\left(\frac{\gamma_{d-1, k_{d-1}}}{n}\right)^{s_d}\frac{\gamma_{d-1, k_{d-1}-1}}{e s_d!} >\frac{n}{e s_d!}\left(\frac{\tilde{\rho}}{4ed\ell}\right)^{(s_d+1)\prod_{j=1}^{d-1}(k_j+2)}.%\geq \Big(\frac{\tilde{\rho}}{4ed\ell}\Big)^{\left(\left(1-\frac{\epsilon}{d}\right)\ell+\frac{3d-2}{d}\right)^d}.
\end{equation}
Using $s_d=(d-\epsilon)\ell-\sum_{j=0}^{d-1}(k_j-1)>0$ again, we have $(s_d+1)+\sum_{j=1}^{d-1}(k_j+2)\leq (d-\epsilon)\ell+3d-2$.
%$\prod_{i=1}^dx_i\leq \left(c/d\right)^{d}$ for $x_i\geq0$ such that $\sum_{i=1}^kx_i= c$.
Apply the inequality of arithmetic and geometric means to the exponent of the right-hand side of \eqref{eq:last step bound} to obtain
$$
(s_d+1)\prod_{j=1}^{d-1}(k_j+2)\leq \left(\left(1-\frac{\epsilon}{d}\right)\ell+\frac{3d-2}{d}\right)^d< \left(1-\frac{2\epsilon}{3d}\right)\ell^d,
$$
where the 2nd inequality holds when $\epsilon>9d/\ell=\Theta\big((\log\log n/\log n)^{-1/d}\big)$. Combining with the fact that $es_d!<(d\ell)^{d\ell}$, we have
$$
\left(\frac{\gamma_{d-1, k_{d-1}+1}}{n}\right)^{s_d}\frac{\gamma_{d-1, k_{d-1}}}{e s_d!}>\frac{n}{(d\ell)^{d\ell}}\left(\frac{\tilde{\rho}}{4ed\ell}\right)^{-(1-2\epsilon/3d)\ell^d}>n\ell^{-(1-\epsilon/2d)\ell^d},
$$
where the 2nd inequality holds when $\epsilon>7d(\log(4ed)-\log\tilde{\rho})/\log\ell=\Theta(1/\log\log n)$. For such $n$, we obtain
\begin{align}
\sum_{k_1=1}^{\lceil s_1\rceil}\cdots\sum_{k_{d-1}=1}^{\lceil s_{d-1}\rceil}\P(A_{d-1, k_{d-1}}\cap E_{d-1, k_{d-1}-1}^c)
&\leq (d\ell)^{d-1}\exp\left(-n\ell^{-(1-\epsilon/2d)\ell^d}\right)\notag\\
%&<(d\ell)^{d-1}\exp\left(-n^{1-\left(1-\frac{\epsilon}{2d}\right)^d}\right) \notag\\
&< (d\ell)^{d-1}\exp\left(-n^{\epsilon/2d}\right) \notag\\
&<\exp\left(-n^{\epsilon/3d}\right). \label{eq: term1}\
\end{align}
In the second inequality, we use the fact that $n\ge \ell^{\ell^d}$ as stated in \eqref{eq:ell n relations}. The last inequality holds when $\epsilon>6d(\log d+\log\log(d\ell))/\log n=\Theta(\log\log\log n/\log n)$.

The proposition follows from \eqref{eq:final-inductive bound}, \eqref{eq: term2} and \eqref{eq: term1}.
\end{proof}

\begin{rmk}
Similar to Proposition \ref{prop:upper bound}, the parameter $\epsilon$ in Proposition \ref{prop:lower bound} is also allowed to depend on $n$ and Proposition \ref{prop:lower bound} is valid as long as $\epsilon>C/\log\log n$ for an appropriate absolute constant $C$. With this choice of $\epsilon$, the lower bound in Theorem \ref{thm:main} is an immediate consequence of Proposition \ref{prop:lower bound}.
\end{rmk}

\section*{Acknowledgement}
We thank Ori Gurel-Gurevich for reading an early manuscript. We are grateful to the anonymous referee for valuable comments.

\end{document}